\documentclass{article}
\usepackage{amsmath, amssymb, amsthm, relsize, tikz-cd}

\theoremstyle{plain}
\newtheorem{theorem}{Theorem}[section]
\newtheorem{lemma}[theorem]{Lemma}

\theoremstyle{definition}

\DeclareMathOperator*{\OmSum}{\mathlarger{\mathlarger{\Omega}}}

\newcommand{\slog}{\text{slog}}
\newcommand{\tet}{\uparrow \uparrow}
\newcommand{\up}{\uparrow}

\begin{document}

\title{Hyper-operations By Unconventional Means}

\author{James David Nixon\\
	JmsNxn92@gmail.com\\}

\maketitle

\begin{abstract}
The author makes use of infinite compositions and a limiting function to construct a $\mathcal{C}^\infty$ tetration function $\mathcal{F}(t) = e \tet t$. As a tetration function, $\mathcal{F}$ satisfies $e^{\mathcal{F}(t)} = \mathcal{F}(t+1)$. Of it, $\mathcal{F}$ takes $(-2,\infty) \to \mathbb{R}$ bijectively with strictly monotone growth, and is continuously differentiable here. We then iterate this construction to derive arbitrary hyper-operations $e\up^k t$. These hyper-operations are $\mathcal{C}^\infty$ strictly monotone bijections of $(\alpha_k,\infty) \to \mathbb{R}$ for $k$ even ($1-k > \alpha_k \ge -k$), and $\mathcal{C}^\infty$ strictly monotone bijections of $\mathbb{R} \to (\alpha_k, \infty)$ for $k$ odd. These hyper-operations satisfy the functional equation $e \up^{k-1} (e \up^k t) = e \up^k (t+1)$ with the initial conditions $e \up^1 t = e^t$ and $e \up^k 0 = 1$.
\end{abstract}



\section{Introduction}\label{sec1}
\setcounter{equation}{0}

The study of tetration can be traced back centuries. The first notable moment being, when Leonhard Euler proved that when $e^{-e} < \alpha < e^{1/e}$; the infinite tower converges \cite{Eul}. Quote unquote, for such $\alpha$,

\[
\alpha^{\displaystyle \alpha^{\displaystyle ...(n\,\text{times})...^{\displaystyle\alpha}}} \to A\,\,\text{as}\,\,n\to\infty
\]

Where $\alpha^A = A$ and $e^{-1} \le A \le e$. This was a monumentous moment in the study of iterated exponentials; and is probably the first truly publishable result in the history of tetration--hell, complex dynamics. Though the subject of tetration remains rather dormant; if we do not count numerous forays into the complex dynamics of the exponential function; it poses itself as a very interesting problem. The author would like to think, at the center of the study of iterated exponentials is the study of tetration functions; and the quest for a good and right solution should be a priority. But as the pulse of the subject is weak, he feels not enough people feel the same way.

A tetration function is simple enough to describe, but proves a very difficult function to construct (if we want a good and right solution). We'll restrict our attention to tetration functions with the base $e$. For bases $b > e^{1/e}$ the result is similar. If $\mathcal{F}$ is a function such that $e^{\mathcal{F}(s)} = \mathcal{F}(s+1)$ and $\mathcal{F}(0) = 1$, then we call $\mathcal{F}$ a tetration function. It provides a similar continuation of the sequence $e, e^e, e^{e^e},...$ as $e^s$ provides a continuation of the sequence $e, ee, eee,...$. In contrast to the exponential function though, it is a much more volatile construction.

One facet of its volatility, tetration functions are highly non-unique. For any tetration function $\mathcal{F}$ the function $\mathcal{F}(s + \sin(2 \pi s))$ is also a tetration function; as is $\mathcal{F}(s+ \theta(s))$ for any $1$-periodic function. Alors, a large problem with tetration, is qualifying a tetration function as unique, or as satisfying some property which characterizes it from other tetration functions.

There are quite a few trivial ways to construct a tetration function, for instance letting $\mathcal{F}(t) = 1+t$ for $t \in [-1,0]$ and extending $\mathcal{F}$ to $\mathbb{R}^+$ using the functional equation provides such a solution. But necessarily such a solution is not differentiable on the natural numbers. One can construct many continuous extensions in a parallel manner.

If we were to ask for a tetration function, we would at least require that it be analytic. Even better, that it be holomorphic on some domain in the complex plane including $\mathbb{R}^+$.

Of that end, Hellmuth Kneser was the first to undoubtedly provide a desirable solution to the tetration equation \cite{Kne}. Insofar as it took the real positive line to itself and was holomorphic. He worked exclusively with the base $b = e$, and managed to construct a function $\mathcal{H}(s) =\, ^s e$ holomorphic in $\mathbb{C}$ excluding the line $(-\infty,-2]$. His construction, although providing a stable solution, was highly esoteric and deeply expressed the difficulty of this problem. Constructing a holomorphic tetration function is no easy feat. To his talent, his sole goal was to construct $h$ such that $h(h(z)) = e^z$ and $h$ was real-valued; and he managed to construct the only true holomorphic tetration.

More recently, numerous attempts have been made to construct a simpler tetration function. There exists quite a few potential candidates for tetration; which exist scattered in the recesses of the internet. Since tetration has not gained widespread recognition as a notable problem--it is difficult to find published papers on the subject. Kneser's own paper \cite{Kne} is in German, and there exists no English translations--but there are synopses and breakdowns. The author's sole resource is to point the reader to The Tetration Forum \cite{Trp}--started by Henryk Trappman, it has grown into a hub of information on Tetration.

The problem with most modern approaches to tetration seem to be that these candidates can be numerically verified, but never rigorously justified to exist or converge. And in contrast; Kneser's construction, which is correct, is a rather laborious numerical procedure. Even more troubling with these flurry of candidates to tetration; proofs of analycity become even more difficult--and in most instances do not exist. Alongside a lack of proof for mere convergence, this can be troubling. If that wasn't bad enough; we can't even prove if two candidate tetration functions equal or not. But sometimes our floating point accuracy can appear to suggest so (or dissuade so).

Nonetheless there is still great headway being made in the field. A treasured example is in the work of Dmitri Kouznetsov \cite{Kou}. The author will not go into detail on Kouznetsov's method but will simply give a rough heuristic as motivation.

Supposing we had a nice function $g(s)$, which has some desirable growth properties, and we took the limiting function,

\[
G_n(s) = \log\log\cdots(n\,\text{times})\cdots\log g(s+n)
\]

Then $e^{G_n(s)} = G_{n-1}(s+1)$. If the limit were to converge as $n\to\infty$, then we would have our tetration function $G_n \to G$. That is, upto a normalization constant $\omega$ where $\mathcal{F}(s) = G(s+\omega)$ to ensure $\mathcal{F}(0) = 1$.

Kouznetsov chose a very wonderful function $g$, such that numerically everything worked out and provides us with a calculator's version of a holomorphic tetration function. The function $g$ was constructed through careful fixed-point analysis; and looks like an exponential sum of terms $e^{nLs}$ for $L$ a complex fixed point of $e^s$. Unfortunately, to rigorously justify convergence proves rather difficult. This becomes a sort of brick-wall. There is no in your hands proof that Kouznetsov's method actually works.

This paper, however, more closely parallels Peter Walker's work in \cite{Wlk}. We'll only be concerned with Real Analysis in this paper; as Walker was. Walker uses the exact trick Kouznetsov uses--and it predates Kouznetsov. Walker's approach was to choose the inverse Abel function of $e^t-1$ as $g$. Now where Walker's paper fails in holomorphy, it makes up for in cold hard $\epsilon/\delta$. This solution is a diffeomorphism of $(-2,\infty)\to\mathbb{R}$; a smooth bijection with inverse.

Par \c{c}a, the goal of this paper is to construct our own $g$, and show the convergence of the above limit--in the manner that Walker did. It needs to work well enough so we can iterate to get higher order hyper-operators. Where, not necessarily that Peter Walker's method would not work in this scenario; it just seems impossible to prove. Where as, the $g$ we choose to construct tetration; can be abstracted to higher orders. We do not need the unnecessarily complex procedure of constructing the inverse Abel function of $e^t -1$.

Our choice of $g$ will be very manufactured, and requires a familiarity with infinite compositions. To that end, we refer the reader to \cite{Nix}, where sufficient conditions are provided for an infinite composition to converge--and a familiarity with the subject is created. We shall not need anything from \cite{Nix}, but it exhibits the nuanced detail of the subject a bit more clearly. We will prove a modified form of the main result of \cite{Nix}; to keep this paper self-contained; but we will give little to no motivating intuition in this paper.

The essential trick is to bypass the difficulty of constructing a tetration function using Walker or Kouznetsov's trick by constructing a function which satisfies a similar functional equation as tetration, and exhibits the same growth properties.

To that end, our first goal is to construct an entire function $\phi(s)$ such that,

\[
\phi(s+1) = e^{s + \phi(s)}
\]

And take $\phi$ to be our $g$ from above. The real novelty of the work is held in constructing $\phi$. But it only really takes us writing out the equation for $\phi$ and justifying convergence. This is more of a taxing process than a difficult one. It is surprisingly simple to construct $\phi$. 

When generalizing to higher order hyper-operations, we will abandon the need for holomorphy. We'll focus exclusively on subsets of $\mathbb{R}$.

\begin{center}
* \quad * \quad *
\end{center}

We introduce briefly the notation for nested compositions, which allows for our construction of $\phi$. We will restrict from full generality, and only care about a subset of types of infinite compositions. Therein, if $h_j(s,z):\mathbb{C}^2 \to \mathbb{C}$ is a sequence of entire functions in both variables, then,

\[
\OmSum_{j=1}^n h_j(s,z) \bullet z = h_1(s,h_2(s,...h_n(s,z)))
\]

Where we are interested in letting $n\to\infty$. The study of infinite compositions is very nuanced, and for that reason constructing $\phi$ will require care. The type of convergence we'll need is one which is a bit simpler than the general case. To wit, we will call,

\[
\phi_n(s) = \OmSum_{j=1}^n e^{s-j+z}\bullet z \Big{|}_{z=0}
\]

Where if $h_j(s,z) = e^{s-j+z}$ then,

\[
\phi_n(s) = h_1(s,h_2(s,...h_n(s,0)))
\]

By design, $e^{s + \phi_n(s)} = \phi_{n+1}(s+1)$, so if this were to converge it would equal our desired function. The essential ingredient in our construction is for all compact disks $\mathcal{P},\mathcal{K} \subset \mathbb{C}$,

\[
\sum_{j=1}^\infty ||h_j(s,z)||_{s \in \mathcal{P}, z \in \mathcal{K}} < \infty
\]

Where $||...||$ is taken to mean the supremum norm.

\begin{center}
* \quad * \quad *
\end{center}

Then with this construction under our belt, we can iterate the procedure. Whereas Walker's construction inherently relies on the similarity between $e^x -1$ and $e^x$, and the ability to construct an Abel function--we don't need this in our case. And we can do what we used from $e \up t = e^t$ to construct $e \up \up t$; to construct $e \up \up \up t$, and so on; without a convenient Abel function. The thesis of this paper is the malleability of the approach to work in many exotic cases. And we don't really need an Abel function--as Walker's case requires.

Instead of needing to find a function $A$ such that,

\[
A(f(t)) = A(t) + 1\\
\]

We use the function $\Phi$,

\[
\Phi(t+1) = e^t f(\Phi(t))\\
\]

Which is easy to construct; and exhibits enough desirable growth properties. Almost analogous to $A$'s growth; a slightly larger growth; but there's more than enough control of the error. So without running the well dry, we'll start the paper.

\section{Constructing $\phi$}\label{sec2}
\setcounter{equation}{0}

The first thing we need to construct $\phi$ is a sort of normality condition. For all $\epsilon > 0$, there exists some $N$, such when $m \ge n  > N$,

\[
|\OmSum_{j=n}^{m} e^{s-j+z} \bullet z| = |e^{\displaystyle s-n+e^{ s-n-1+...^{e^{s-m+z}}}}| < \epsilon
\]

For $|z|<1$, and $s$ residing in some compact disk within $\mathbb{C}$. This then implies as we let $m\to\infty$, the tail of the infinite composition stays bounded. Forthwith, the infinite composition becomes a normal family, and proving convergence becomes simpler. 

The direct analogy we can make to this is with sums and products. Suppose we have a sequence of numbers $a_j \in \mathbb{C}$ and we wish to show their sum converges. Well the direct method is to prove,

\[
|\sum_{j=n}^m a_j| < \epsilon\\
\]

This is similar with products; for a sequence of numbers $b_j \to 1$; the direct method is to prove,

\[
|\prod_{j=n}^m b_j -1| < \epsilon\\
\]

So, in the compositional case, we're doing something very similar. We're trying to bound the tail of the infinite composition. The sort of bounds we get are, not exactly but close to, of the form,

\[
|\OmSum_{j=n}^{m} e^{s-j+z} \bullet z| \le \sum_{j=n}^m |e^{s-j+z}|
\]

And in this, we are comparing the tail of the composition to the tail of a sum. It is a bit more subtle how we actually do this; supremum norms across different domains arise. The lemma may be a tad cryptic; but this is a good way to initialize the idea. We provide a quick proof of this.

\begin{lemma}\label{lma2A}
For a compact disk $\mathcal{P} \subset\mathbb{C}$ and $|z| \le 1$: for all $\epsilon > 0$, there exists some $N$, such when $m \ge n  > N$

\[
||\OmSum_{j=n}^{m} e^{s-j+z} \bullet z||_{\mathcal{P},|z|\le 1} < \epsilon
\]
\end{lemma}

\begin{proof}
Let $|z| \le 1$ and $s \in \mathcal{P}$ be a compact disk in $\mathbb{C}$.  Set $h_j(s,z) = e^{s-j+z}$ and set $||h_j (s, z)||_{s\in\mathcal{P}, |z| \le 1} = \rho_j$. Pick $1>\epsilon > 0$, and choose $N$ large enough so when $n > N$,

\[
\rho_n < \epsilon
\]

Denote: $\phi_{nm}(s, z) =\OmSum_{j=n}^m h_j (s, z) \bullet z = h_n(s, h_{n+1}(s, ...h_m(s, z)))$. We go by induction on the difference $m-n = k$; which counts how many exponentials appear. When $k=0$ then,

\[
||\phi_{nn}(s,z)||_{|z| < 1, s \in \mathcal{P}} = ||h_n(s,z)||_{|z| < 1, s \in \mathcal{P}} = \rho_n < \epsilon
\]

Assume the result holds for $m-n < k$, we show it holds for $m-n = k$. Observe,

\begin{eqnarray*}
||\phi_{nm}(s,z)||_{|z| < 1, s \in \mathcal{P}} &=& ||h_n(s,\phi_{(n+1)m}(s,z))||_{|z| < 1, s \in \mathcal{P}}\\
&\le& ||h_n(s,z)||_{|z|<1, s \in \mathcal{P}}\\
&=& \rho_n < \epsilon\\
\end{eqnarray*}

Which follows by the induction hypothesis because $|\phi_{(n+1)m}(s,z)| < \epsilon < 1$; i.e: $m - n - 1 < k$.
\end{proof}

The next step is to observe that $\OmSum_{j=1}^m h_j(s,z)\bullet z$ is a normal family as $m\to\infty$, for $|z|<1$ and $s \in \mathcal{P}$, an arbitrary compact disk. This follows because the tail of this composition is bounded. For $m > N$,

\[
\OmSum_{j=1}^m h_j(s,z) = \OmSum_{j=1}^N h_j(s,z)\bullet \OmSum_{j=N+1}^m h_j(s,z)\,\bullet z\\
\]

But $||\OmSum_{j=N+1}^m h_j(s,z)\,\bullet z|| < \epsilon$. So as $m$ grows the function stays within the neighborhood. Therefore we can say $||\OmSum_{j=1}^m h_j(s,z)||_{|z|<1, s \in \mathcal{P}} < M$ for all $m$.

The second way to ideate this, is to, again, compare it to a sum. What we have done is made a comparison,

\[
|\OmSum_{j=1}^\infty e^{s-j+z}\bullet z| \le \sum_{j=1}^\infty |e^{s-j+z}|
\]

Where on the right hand side $|...|$ is a supremum norm across some compact set, and on the left hand side $|...|$ is across a compact set. This is a nice way to think about it; where because the sum converges, so does the infinite composition. It may seem strange; but the idea is,

\[
|\OmSum_{j=1}^m e^{s-j+z}\bullet z - \OmSum_{j=1}^n e^{s-j+z}\bullet z| \le \sum_{j=n}^m |e^{s-j+z}| < \epsilon\\
\]

Where on the right hand side $|...|$ is a supremum norm for some compact set, and on the left hand side $|...|$ is across a compact set. Which is very similar to the sum and product case, in which:

\begin{eqnarray*}
|\sum_{j=1}^m a_j - \sum_{j=1}^n a_j| &\le& \sum_{j=n}^m |a_j| < \epsilon\\
|\prod_{j=1}^m b_j - \prod_{j=1}^n b_j| &\le& A\sum_{j=n}^m |b_j-1| < \epsilon\,\,\text{for some}\,\,A \in \mathbb{R}^+\\
\end{eqnarray*}

From this we can prove our infinite composition converges, and construct our entire function $\phi(s)$.

\begin{theorem}\label{thm2A}
The expression

\[
\OmSum_{j=1}^\infty e^{s-j+z} \bullet z \Big{|}_{z=0} = \phi(s)
\]

is an entire function satisfying the identity $e^{s + \phi(s)} = \phi(s+1)$.
\end{theorem}

\begin{proof}
Since $\phi_m(s,z) = \OmSum_{j=1}^m e^{s-j+z}\bullet z$ are a normal family; there is some constant $M \in \mathbb{R}^+$ such,

\[
||\frac{d^k}{dz^k} \phi_m(s,z) ||_{|z|<\frac{1}{2}, s \in\mathcal{P}} \le M2^{k+1} \cdot k! 
\]

We can achieve this through Cauchy's Integral Theorem; and a couple uses of the supremum norm; because,

\begin{eqnarray*}
\frac{d^k}{dz^k} \phi_m(s,z) &=& \frac{k!}{2\pi i} \int_{|\xi|=1} \frac{\phi_m(s,\xi)}{(\xi - z)^{k+1}}\,d\xi\\
||\frac{d^k}{dz^k} \phi_m(s,z) ||_{|z|<\frac{1}{2}, s \in\mathcal{P}} &\le& \frac{k!}{2\pi} \int_{|\xi|=1} 2^{k+1}M\,d\xi\\
\end{eqnarray*}

When $|z| \le 1/2$ and $|\xi| = 1$, $|\xi - z| \ge \frac{1}{2}$. Secondly, using Taylor's theorem; and expanding $\phi_{m+1}(s,z) = \phi_m(s,e^{s-m-1+z})$ about $\phi_m(s,z)$'s Taylor's series about $z$:

\begin{eqnarray*}
\phi_{m+1}(s,z) - \phi_m(s,z) &=& \phi_m(s,e^{s-m-1+z}) - \phi_m(s,z)\\
&=& \sum_{k=1}^\infty \frac{d^k}{dz^k} \phi_m(s,z) \frac{(e^{s-m-1+z} - z)^k}{k!}\\
&=& (e^{s-m-1+z} - z) \sum_{k=1}^\infty \frac{d^k}{dz^k} \phi_m(s,z) \frac{(e^{s-m-1+z} - z)^{k-1}}{k!}\\
\end{eqnarray*}

This series converges for at least $|z| < \delta$ and $m>N$ large enough. We don't care about $\delta$ because we set $z=0$. Then,

\[
||\phi_{m+1}(s,0) - \phi_m(s,0)||_{s \in \mathcal{P}} \le ||e^{s-m+1}||_{s\in\mathcal{P}}\sum_{k=1}^\infty M 2^{k+1}||e^{s-m+1}||_{s\in\mathcal{P}}^{k-1}\\
\]

The series on the right can be bounded by some $C \in \mathbb{R}^+$; because for large enough $m>N$ the term $||e^{s-m+1}||_{s\in\mathcal{P}} < \frac{1}{2}$. Applying the bounds,

\[
||\phi_{m+1}(s,0) - \phi_m(s,0)||_{s \in \mathcal{P}} \le C ||e^{s-m-1}||_{s \in \mathcal{P}} = A e^{-m}
\]

For some $A \in \mathbb{R}^+$. Assume,

\[
\sum_{j=n}^{m-1} e^{-j} < \frac{\epsilon}{A}\\
\]

Then,

\begin{eqnarray*}
||\phi_m(s,0) - \phi_n(s,0)||_{s \in \mathcal{P}} &\le& \sum_{j=n}^{m-1}||\phi_{j+1}(s,0) - \phi_j(s,0)||_{s \in \mathcal{P}}\\
&\le& \sum_{j=n}^{m-1} A e^{-j}\\
&<&\epsilon\\
\end{eqnarray*}

And we can see the telescoping series converges and $\phi_m(s)$ must be uniformly convergent for $s \in \mathcal{P}$, and therefore defines a holomorphic function $\phi(s)$ as $m\to\infty$. Naturally $e^{s + \phi_m(s)} = \phi_{m+1}(s+1)$, and so therefore the functional equation is satisfied. Since $\mathcal{P}$ was an arbitrary compact set in $\mathbb{C}$, we know $\phi$ is entire.
\end{proof}

\section{The correction term $\tau$}\label{sec3}
\setcounter{equation}{0}

The main philosophy of our approach to constructing tetration is to add a corrective term to $\phi$ such that it becomes a tetration function. The function $\phi$ already looks very close to tetration, satisfying a similar functional equation,

\[
\phi(s+1) = e^{s + \phi(s)}
\]

We will introduce a sequence of correction terms as follows:

\[
\log \log \cdots (n\,\text{times}) \cdots \log \phi(s+n) = \phi(s) + \tau_n(s)
\]

Where inductively, starting with $\tau_1(s) = s$ and $\tau_0(s)=0$; $\tau_n$ can be defined,

\begin{eqnarray*}
\tau_{n+1}(s) &=& \log\big{(}\phi(s+1) + \tau_n(s+1)\big{)} - \phi(s)\\
&=& \log\phi(s+1) + \log\big{(}1 + \frac{\tau_n(s+1)}{\phi(s+1)}\big{)} - \phi(s)\\
&=& s + \phi(s) + \log\big{(}1 + \frac{\tau_n(s+1)}{\phi(s+1)}\big{)} - \phi(s)\\
&=& s + \log(1 + \frac{\tau_n(s+1)}{\phi(s+1)})\\
\end{eqnarray*}

Our choice of $\log$ is defined implicitly by the relation,

\[
e^{\phi(s) + \tau_{n+1}(s)} = \phi(s+1) + \tau_n(s+1)
\]

And the restriction that $\tau_n(s)$ is real on the real-line. The first thing to note, is that for $s =t \in \mathbb{R}^+$, the sequence of functions $\tau_n$ converge uniformly on bounded intervals greater than $T$. This is because $0 < \tau_{n+1}(t+1) / \phi(t+1) < 1$ and using the relation $|\log(1+A) - \log(1+B)| \le |A - B|$,

\begin{eqnarray*}
|\tau_{n+1}(t) - \tau_n(t)| &\le& | \log(1 + \frac{\tau_{n}(t+1)}{\phi(t+1)}) - \log(1+ \frac{\tau_{n-1}(t+1)}{\phi(t+1)})|\\
&\le& \frac{1}{\phi(t+1)} |\tau_{n}(t+1) - \tau_{n-1}(t+1)|\\
&\le& \frac{1}{\phi(t+1)\phi(t+2)}|\tau_{n-1}(t+2) - \tau_{n-2}(t+2)|\\
&\vdots&\\
&\le& \frac{1}{\prod_{k=1}^n \phi(t+k)}|\tau_1(t + n) - \tau_0(t+n)|\\
\end{eqnarray*}

Recalling that $\tau_1(t) = t$ and $\tau_0(t) = 0$ then,

\[
|\tau_{n+1}(t) - \tau_n(t)| \le \frac{t+n}{\prod_{k=1}^{n}\phi(t+k)} 
\]

And here $\phi(t)$ is monotone increasing and unbounded, so for some $T$ with $t > T$ it is $\phi(t)> \lambda > 1$. Now to show the telescoping sum converges uniformly on bounded intervals; pick an interval $\mathcal{I}$; and pick $n,m > N$ such,

\[
\sum_{j=n}^{m-1} \frac{||t+j||_{\mathcal{I}}}{\lambda^j} < \epsilon\\
\]

Then,

\begin{eqnarray*}
||\tau_m(t) - \tau_n(t)||_{\mathcal{I}} &\le& \sum_{j=n}^{m-1} ||\tau_{j+1}(t) - \tau_j(t)||_{\mathcal{I}}\\
&\le& \sum_{j=n}^{m-1} \frac{||t+j||_{\mathcal{I}}}{\lambda^j}\\
&<& \epsilon\\
\end{eqnarray*}

And we are given a function $\tau : \mathbb{R}_{t > T}^+ \to \mathbb{R}^+$ such that,

\[
\widetilde{\mathcal{F}}(t) = \phi(t) + \tau(t)
\]

And $\widetilde{\mathcal{F}}$ is a tetration function, albeit not yet normalized to $\widetilde{\mathcal{F}}(0)=1$, but there is an appropriate $\omega$ such that $\mathcal{F}(t) = \phi(t+ \omega) + \tau(t+\omega)$ is a true tetration function. Also by taking logarithms, the domain can be extended to its maximal $(-2, \infty)$. Call this function $\mathcal{F}$.

To assure that this construction hits no singularities along the way, we need that the derivative is monotone. Going through the same motions as above one can derive that $\mathcal{F}'$ is a continuous function and that $\tau_n' \to \tau'$ uniformly on bounded intervals. This is really no different then what we've already written. The function,

\[
\tau'_{n+1}(t) = 1 + \Big{(}\frac{1}{1 + \frac{\tau_n(t+1)}{\phi(t+1)}}\Big{)}\Big{(}\frac{\tau'_n(t+1)}{\phi(t+1)} - \frac{\tau_n(t+1)}{\phi(t+1)^2}\phi'(t+1)\Big{)}\\
\]

Grinding the gears of this expression we get $|\tau'_{n+1} - \tau'_n|$ is a convergent series of the same form as above. This is more of a task than a problem. It is left to the reader.

We need that $\mathcal{F}'(t) > 0$ for all $t \in (-2, \infty)$. From the expression above, $\tau'(t)-1\to 0$ as $t\to\infty$. So, eventually $\tau'(t) >0$ for some $t \ge T$. It is no hard fact to notice $\phi'(t) > 0$ for all $t \in \mathbb{R}$. Ergo, for $t\ge T$,

\[
F'(t) = \phi'(t+\omega) + \tau'(t+\omega) > 0\\
\]

Therefore, since,

\[
\mathcal{F}'(t-1) = \frac{\mathcal{F}'(t)}{\mathcal{F}(t)} > 0\\
\]

Thereby $\mathcal{F}'(t) > 0$ for $t \ge T-1$. By infinite descent we must have $\mathcal{F}'(t) > 0$ everywhere $\mathcal{F}(t+1) > 0$ which is $t>-2$. Therefore of this nature we have a differentiable inverse $\mathcal{A} = \mathcal{F}^{-1}(t) : \mathbb{R} \to (-2,\infty)$. In laymen's terms, amongst the jargon of people who study tetration; one calls this the super-logarithm. It is a continuously differentiable Abel function of $e^t$. En drame,

\[
\mathcal{A}(e^t) = \mathcal{A}(t) + 1\\
\]

These facts will be reinforced throughout this paper. Nonetheless it helps to introduce them when they can be conveniently introduced. We state this less than drastic theorem below.

\begin{theorem}\label{thm3A}
For some $\omega \in \mathbb{R}$ there exists a continuously differentiable tetration function $\mathcal{F}(t):(-2,\infty) \to \mathbb{R}$ such that $\mathcal{F}'(t) > 0$ and $\mathcal{F}$ is a bijection. This tetration function $\mathcal{F}(t)$ can be expressed as,

\[
\mathcal{F}(t) = \lim_{n\to\infty}\log \log\cdots(n\,\text{times})\cdots\log \phi(t+\omega+n)
\]

Where,

\[
\phi(t) = \OmSum_{j=1}^\infty e^{t-j+z}\bullet z \Big{|}_{z=0}\\
\]
\end{theorem}

This provides us with a continuously differentiable tetration function $\mathcal{F}$ defined for $(-2,\infty)$, but it sadly says nothing of the case for complex numbers. This proves to be a much more exhausting challenge--the author could not resolve it. Instead the author will show that $\tau$ is infinitely differentiable.

\section{Showing infinite differentiability}\label{sec4}
\setcounter{equation}{0}

The infinite differentiability of $\tau$ is a fun exercise in compositional analysis. We really have to pull out all the stops; the idea is inductive; and the proofs are mostly just lemmas. It's a good exercise in bounding compositions with sums. It's a good exercise in exotic iterative procedures, and the quickness of $\Omega$-notation. For that, the work to follow may seem terse; or abnormal. The author will be diligent in explaining every conclusion to the best of his ability. But he'd like it if the reader could appreciate the malleability of the work.

Since $\phi$ is so well behaved, and $\frac{1}{\phi}$ is smaller than any iterate of the exponential, we can make cake work of this. If we take each $\tau_m$ then,

\[
|\tau_m(t) - t| \le \sum_{j=1}^m \frac{t+j}{\prod_{c=1}^j \phi(t+c)}\\
\]

For all $t \ge T$. This relationship is also satisfied for all derivatives. We'll write this a tad simpler,

\[
|\frac{d^l}{dt^l} \tau_{m}(t) - \frac{d^l}{dt^l} t| \le A_{lm} e^{-e^t}\\
\]

This can be shown by induction on $m$ for each $l$. To visualize the argument, use a sequence of L'H\^{o}pital/Bernoulli comparisons, and prove them inductively on $m$ using the functional equation of $\tau_m = t + \log(1+\frac{\tau_{m-1}(t+1)}{\phi(t+1)})$,

\begin{eqnarray*}
\frac{\tau_m(t)}{t} &=& 1 + \mathcal{O}(e^{-e^t})\\
\tau_m'(t) &=& 1 + \mathcal{O}(e^{-e^t})\\
\tau''_m(t) &=& \mathcal{O}(e^{-e^t})\\
&\vdots&\\
\tau^{(l)}_m(t) &=& \mathcal{O}(e^{-e^t})\\
\end{eqnarray*}

This isn't far fetched considering that $\frac{1}{\phi(t)} \le \frac{1}{\exp^{\circ n}(t)}$, and the decay is uniform as we take derivatives (in fact it's more than Schwartz). This sequence of bounds is pretty weak; but stronger would be over-kill. As they stand, how strong these bounds are, isn't so much needed either; but it reduces some of the complexity of the problem. 

Continuing:

\[
\sum_{c=1}^\infty ||\tau_m^{(l)}(t+c) - \frac{d^l}{dt^l} (t+c)||_{t\ge T} \le \sum_{c=1}^\infty A_{lm}e^{-e^{c}} < \infty\\
\]

For $T$ large enough. Since we know $\tau$ is differentiable, we only care about $l>1$. Making this simpler. For $l > 1$,

\[
\sum_{c=1}^\infty ||\tau_m^{(l)}(t+c)||_{t\ge T} < \infty\\
\]

Now we remember that,

\[
\tau_{m+1}(t) = t + \log\big{(}1 + \frac{\tau_m(t+1)}{\phi(t+1)}\big{)}\\
\]

If we differentiate this $l$ times; there is a function $F_l$ in $l+1$ variables, such that,

\[
\tau^{(l)}_m(t) = F_l(t+1,\tau_{m-1}(t+1), \tau_{m-1}'(t+1),...,\tau^{(l)}_{m-1}(t+1))\\
\]

And the sequence $F_l$ is generated by,

\[
F_{l+1}(t,x_0,...,x_{l+1}) = \frac{\partial F_l}{\partial t} + \sum_{d=0}^{l} x_{d+1}\frac{\partial F_l}{\partial x_d}(t,x_0,...,x_l)\\
\]

Starting with,

\[
F_0(t,x_0) = t + \log\big{(}1 + \frac{x_0}{\phi(t+1)}\big{)}\\
\]

It's up to the reader to check that $F_0(t) - t = \mathcal{O}(e^{-e^t})$, $F_1(t) -1 = \mathcal{O}(e^{-e^t})$, and $F_l(t)=\mathcal{O}(e^{-e^t})$ for $l>1$. The philosophy of our approach is to boil all our questions into the well performed behaviour of $F_l$. Using infinite compositions, we arrive at the iterative formula:

\[
\tau_m^{(l)}(t) = \OmSum_{c=1}^m F_l(t+c,\tau_{m-c}(t+c),...,\tau_{m-c}^{(l-1)}(t+c),x)\bullet x \Big{|}_{x=0}
\]

Notice that this expression only uses the derivatives up to $l-1$, so if this infinite composition converges when we take $m\to\infty$, we're all done. Insofar as, a proof by strong induction on $l$ would suffice. This infinite composition looks a tad different than what we are used to so far--but it's not much of a different beast.

We can first note that we still have a summability criterion. For all $X \in \mathbb{R}^+$ and for large enough $t\ge T$; with $l > 1$,

\[
\lim_{m\to\infty}\sum_{c=1}^m ||F_l(t+c,\tau_{m-c}(t+c),...,\tau_{m-c}^{(l-1)}(t+c),x)||_{t\ge T, X \ge x \ge 0} < \infty
\]

And from this we can derive a normality theorem--as we've done before. And then the rest differs very little from before. We show this iteration converges in the next proof.

\begin{theorem}
The function $\tau$ is infinitely differentiable.
\end{theorem}

\begin{proof}
Let's go by induction on the order of the derivative. Start by assuming that, for all $j<l$ that $\frac{d^j}{dt^j}\tau(t+\omega)$ is continuous on $(-2,\infty)$. And additionally, that there is some $A_j$ such that,

\[
||\frac{d^j\tau_m}{dt^j}- \frac{d^jt}{dt^j}||_{t\ge T} \le A_j e^{-e^t}\\
\]

And,

\[
\sum_{m=1}^\infty ||\tau_{m+1}^{(j)} - \tau_m^{(j)}||_{t\ge T} < \infty\\
\]

For all $T> \omega-2$. Choose $T$ large enough so that, $0\le |\tau^{(j)}_m(t) - t^{(j)}| \le 1$. Now, as before, we compare the composition to a sum,

\[
||\OmSum_{p=1}^m q_p(t,x)\bullet x|| \le \sum_{p=1}^m ||q_p(t,x)||
\]

Where the supremum norms are across different sets on either side of the inequality. Insofar,

\begin{eqnarray*}
&&\Big{|}\Big{|}\OmSum_{c=1}^m F_l(t+c,\tau_{m-c}(t+c),...,\tau_{m-c}^{(l-1)}(t+c),x)\bullet x \Big{|}\Big{|}_{t\ge T,1\ge x \ge 0} \le\\
&\le& \sum_{c=1}^m ||F_l(t+c,\tau_{m-c}(t+c),...,\tau_{m-c}^{(l-1)}(t+c),x)||_{t\ge T, X \ge x \ge 0} < \infty\\
\end{eqnarray*}

Since the right hand side is bounded for all $m$; we get the bound,

\begin{eqnarray*}
||\tau^{(l)}_m(t)||_{t\ge T} \le \Big{|}\Big{|}\OmSum_{c=1}^m F_l(t+c,\tau_{m-c}(t+c),...,\tau_{m-c}^{(l-1)}(t+c),x)\bullet x \Big{|}\Big{|}_{t\ge T,1\ge x \ge 0} &\le& M\\
\end{eqnarray*}

This tells us that our sequence $\tau_m^{(l)}$ is bounded and normal for $t \ge T$. We can strengthen this for large enough $t \ge T$ so that $M=1$, $||\tau^{(l)}_m(t)||_{t\ge T} \le 1$ for all $m$. We are going to use upper bounds from the mean-value theorem; but in $l+1$ variables; and in-order to do so we have to derive a bound on the partial derivatives of $F_l$. Recalling that $l>1$; $F_l(t) \to 0$ compactly; we have uniform decay to $0$ of the partial derivatives in $x_j$. So we are given the bounds, for large enough $T$,

\begin{eqnarray*}
\Big{|}\Big{|}\frac{\partial F_l}{\partial x_j}(t,x_0,...,x_l) \Big{|}\Big{|}_{t\ge T, 0 \le |x_0 - t|, |x_1 - 1|,x_2,..., x_l \le 1} &\le& \lambda_j < 1\\
\end{eqnarray*}

For some sequence of numbers $0 < \lambda_0, \lambda_1,...,\lambda_l< 1$. Provided $x_j$ and $x_j'$ are in the compact set $0 \le |x_j - t^{(j)}| \le 1$; we arrive at the bound,

\[
\Big{|}\Big{|}F_l(t,x_0,...,x_l) -F_l(t,x_0',...,x_l') \Big{|}\Big{|}_{t\ge T} \le \sum_{j=0}^l \lambda_j |x_j - x_j'|\\
\]

Since $0 \le || \tau_m^{(l)}|| \le 1$ sits happily as a normal family now; and we already know $0 \le ||\tau_m^{(j)} - t^{(j)}|| \le 1$, we can use the bound:

\begin{eqnarray*}
&&||\tau^{(l)}_{m+1}(t) - \tau^{(l)}_m(t)||_{t\ge T} \le\\ 
&\le&||F_l(t+1, \tau_m(t+1),...,\tau^{(l)}_{m}(t+1)) - F_l(t+1, \tau_{m-1}(t+1),...,\tau^{(l)}_{m-1}(t+1))||\\
&\le& \sum_{j=0}^{l} \lambda_j ||\tau_m^{(j)}(t+1) - \tau_{m-1}^{(j)}(t+1)||_{t \ge T}\\
\end{eqnarray*}

Which can be re-written, by continuing the iteration, remembering $\tau_1^{(l)} = \tau_0^{(l)} = 0$ for $l > 1$,

\begin{eqnarray*}
||\tau^{(l)}_{m+1}(t) - \tau^{(l)}_m(t)||_{t\ge T} &\le& \sum_{c=0}^{m-2} \sum_{j=0}^{l-1} \lambda_j \lambda_l^c||\tau_{m-c}^{(j)}(t+c+1) - \tau_{m-c-1}^{(j)}(t+c+1)||_{t \ge T}\\ 
&+& \lambda_l^{m-1}||\tau_2^{(l)}(t+m) - \tau_1^{(l)}(t+m)||_{t\ge T}\\
\end{eqnarray*}

Par quoi, this is summable by the induction hypothesis. To visualize, compare $||\tau_m^{(j)}(t) - \tau_{m-1}^{(j)}(t)||_{t\ge T} $ to $\mathcal{O}(q^m)$ for $0 < q < 1$--and break out a proof showing the Cauchy product of summable geometric sequences $a_m = \lambda_l^m$, $b_m = \mathcal{O}(q^m)$ is itself summable. And conjoin this with the fact $\sum_c ||\tau_2^{(l)}(t+c) - \tau_1^{(l)}(t+c)||_{t\ge T} < \infty$. Concluding the proof.

\end{proof}

This is a very brief proof, and the language was chosen rather abstractly. This language will help us immensely in what follows. We want to be able to recycle this proof for higher order hyper-operations. It's not exceptionally hard to construct a proof of $\mathcal{C}^\infty$ for $e \tet t$; but choosing a proof which generalizes well; allows us to iterate this procedure into higher order hyper-operators.  

\section{Constructing arbitrary auxiliary functions}\label{sec5}
\setcounter{equation}{0}

In this section we're going to start the inductive step of this paper. Everything up to this point, has been a base step of an inductive proof. Assume the construction,

\[
e \up^j t
\]

For $j < k$ such that the following conditions are satisfied,

\begin{enumerate}
\item $e \up^j 0 = 1$\\
\item If $j$ is even, $e \up^j t : (\alpha_j,\infty) \to \mathbb{R}$ bijectively for some $-j\le \alpha_k \le 1-j$\\
\item If $j$ is odd, $e \up^j t : \mathbb{R} \to (\alpha_j, \infty)$ bijectively\\
\item $e \up^j t$ is infinitely differentiable in $t$\\
\item $\frac{d}{dt}e \up^j t > 0$\\
\item $e\up^j t$ is a diffeomorphism\\
\item $e \up^{j-1} e \up^j t = e \up^j (t+1)$\\
\end{enumerate}

Assuming this for $j<k$, the rest of this paper is about constructing $e \up^k t$. The method we'll use is nearly identical to the beginning of this paper and the construction of $e \tet t$. We ask that the reader take care to notice the isomorphism between the construction of the function in the following theorem, and our construction of $\phi$. This new function will not be holomorphic--but otherwise we'll be given the same construction essentially, but we'll have to work a bit harder to get differentiability.

So to begin, we want to construct an auxiliary function $\Phi(t)$ such that,

\[
\Phi(t+1) = e^t \cdot e \up^{k-1} \Phi(t)\\
\]

\begin{theorem}
The function,

\[
\Phi(t) = \OmSum_{j=1}^\infty e^{t-j} e \up^{k-1} x \bullet x \Big{|}_{x=0}\\
\]

Converges to an infinitely differentiable function $\mathbb{R} \to \mathbb{R}$ such that,

\[
\Phi(t+1) = e^t \cdot e \up^{k-1} \Phi(t)\\
\]
\end{theorem}

\begin{proof}
To begin, we'll gather our normality theorems. For all compact intervals $\mathcal{I} \subset \mathbb{R}$, there exists $N$ such for $n,m > N$; The term,

\[
\Big{|}\Big{|}\OmSum_{j=n}^m e^{t-j} e \up^{k-1} x\Big{|}\Big{|}_{t \in \mathcal{I},0\le x \le 1} < \epsilon\\
\]

Set $h_j(t,x) = e^{t-j}e \up^{k-1} x$ and set $||h_j (t, x)||_{t\in\mathcal{I}, 0\le x \le 1} = \rho_j$. Pick $1>\epsilon > 0$, and choose $N$ large enough so when $n > N$,

\[
\rho_n < \epsilon
\]

Denote: $\Phi_{nm}(t, x) =\OmSum_{j=n}^m h_j (t, x) \bullet x = h_n(t, h_{n+1}(t, ...h_m(t, x)))$. We go by induction on the difference $m-n = k$; which counts how many compositions appear. When $k=0$ then,

\[
||\Phi_{nn}(t,x)||_{\mathcal{I},0\le x \le 1} = ||h_n(t,x)||_{\mathcal{I},0\le x \le 1} = \rho_n < \epsilon
\]

Assume the result holds for $m-n < k$, we show it holds for $m-n = k$. Observe,

\begin{eqnarray*}
||\Phi_{nm}(t,x)||_{\mathcal{I}, 0\le x \le 1} &=& ||h_n(t,\Phi_{(n+1)m}(t,x))||_{\mathcal{I}, 0\le x \le 1}\\
&\le& ||h_n(t,x)||_{\mathcal{I}, 0\le x \le 1}\\
&=& \rho_n < \epsilon\\
\end{eqnarray*}

Which follows by the induction hypothesis because $|\Phi_{(n+1)m}(t,x)| < \epsilon < 1$; i.e: $m - n - 1 < k$. And from this,

\[
||\OmSum_{j=1}^m e^{t-j} e \up^{k-1} x \bullet x ||_{\mathcal{I}, 0\le x \le 1} < M\\
\]

For some $M$. And equally so,

\begin{eqnarray*}
||\frac{d}{dx}\OmSum_{j=1}^m e^{t-j} e \up^{k-1} x \bullet x||_{\mathcal{I}, 0\le x \le 1} &=& \prod_{j=1}^m \big{|}\big{|}e^{t-j}e\up^{k-1} \OmSum_{c=j+1}^m e^{t-c} e \up^{k-1} x \bullet x\big{|}\big{|}_{\mathcal{I}, 0\le x \le 1}\\
&\to& 0\,\,\text{as}\,\,m\to\infty\\
&<& M\\
\end{eqnarray*}

If we choose $M$ large enough. Now, applying the mean value theorem,

\begin{eqnarray*}
||\Phi_{m+1}(t,0) - \Phi_m(t,0)||_{t\in\mathcal{I}} &=& \Big{|}\Big{|}\Phi_m(t,e^{t-m-1}) - \Phi_m(t,0)\Big{|}\Big{|}_{t\in\mathcal{I}}\\
&\le& M e^{t-m-1}\\
\end{eqnarray*}

Which concludes the proof of convergence. To derive infinite differentiability, we go by induction for $j<l$. We can superficially show that,

\[
\Phi^{(l)}(t) = F(t-1,\Phi(t-1),\Phi'(t-1),...,\Phi^{(l)}(t-1))\\
\]

And, if we start with this relationship, and iterate it:

\[
\Phi^{(l)}(t) = \OmSum_{j=1}^\infty F(t-j,\Phi(t-j),\Phi'(t-j),...,\Phi^{(l-1)}(t-j), x) \bullet x \Big{|}_{x=0}\\
\]

Which is simply continuing the iteration towards negative infinity. Notice this has no mention of the $l$'th derivative in this formula, forcing our function to be differentiable $l$ times if it's differentiable for $j<l$--given this infinite composition converges. We leave it to the reader to check; we've done these proofs three times already. It suffices to show,

\[
\sum_{j=1}^\infty ||F(t-j,\Phi(t-j),\Phi'(t-j),..., \Phi^{(l-1)}(t-j), x)||_{t \in \mathcal{I}, X \ge x \ge 0} < \infty\\
\]
\end{proof}

\section{Getting a continuous hyper-operator}\label{sec6}
\setcounter{equation}{0}

We are going to perform the iteration we performed for tetration. We're going to iterate a bunch of $\log$'s--but this time they aren't really $\log$'s. We have a sequence of functions $0 \le j < k$,

\[
\mathcal{A}_{j}(t) = \big{(}e \up^j t\big{)}^{-1}\\
\]

Where this is functional inversion, and not multiplicative inversion. Each of these $\mathcal{A}_j(t)$ grow slower than the last, and each grows slower than $\log(t)$. They grow much, much, slower; we won't need this though, $\log$ suffices. Let's write,

\[
\Phi(t) + \Lambda_n(t) = \mathcal{A}_{k-1} \mathcal{A}_{k-1} \cdots(n\,\text{times})\cdots \mathcal{A}_{k-1}(\Phi(t+n))\\
\]

Which satisfies,

\[
e \up^{k-1} \Phi(t) + \Lambda_n(t) = \Phi(t+1) + \Lambda_{n-1}(t+1)\\
\]

Starting from the iteration $\Lambda_0(t) = 0$. Each of these $\Lambda_n$ are greater than zero because,

\begin{eqnarray*}
\Phi(t+n) &=& e^{t+n}e \up^{k-1} e^{t+n-1}e \up^{k-1} \cdots (n\,\text{times}) \cdots e^t e \up^{k-1} \Phi(t)\\
&>& e \up^{k-1} e \up^{k-1} \cdots (n\,\text{times}) \cdots e \up^{k-1} \Phi(t)\\
\mathcal{A}_{k-1}^{\circ n} \Phi(t+n) &>& \Phi(t)\\
\end{eqnarray*}

for $t > 0$. Which tells us that $\Lambda$ will not converge trivially to a constant if it does converge. Now, the first thing is that, since $e \up^{k-1} t$ grows faster than any iterated exponential, it's inverse grows super slow, and its derivative tends to zero at least like $1/t$. Therefore, if $t_1 > t_0 > T$,

\[
|\mathcal{A}_{k-1}(t_1) - \mathcal{A}_{k-1}(t_0)| \le \frac{1}{t_0}|t_1 - t_0|
\]

And we are essentially done, because,

\begin{eqnarray*}
|\Lambda_{m+1}(t) - \Lambda_{m}(t)| &=& |\mathcal{A}_{k-1}(\Phi(t+1) + \Lambda_{m}(t+1)) - \mathcal{A}_{k-1}(\Phi(t+1) + \Lambda_{m-1}(t+1))|\\
&\le&\Phi(t+1)^{-1}|\Lambda_{m}(t+1) - \Lambda_{m-1}(t+1)|\\
&\le& \big{(}\prod_{j=1}^m  \Phi(t+j)^{-1}\big{)} | \mathcal{A}_{k-1}(\Phi(t+m+1)) - \Phi(t+m)|\\
&\le& \big{(}\prod_{j=1}^m \Phi(t+j)^{-1}\big{)} |t+m|\\
\end{eqnarray*}

Because,

\[
\mathcal{A}_{k-1}(\Phi(t+1)) = \mathcal{A}_{k-1}(e^t e \up^{k-1}\Phi(t)) \le \mathcal{A}_{k-1}(e \up^{k-1}\Phi(t) + t) = \Phi(t) + t\\
\]

This is enough to derive continuity of the limit for $t > T$, taking iterated $\mathcal{A}_{k-1}$ takes us to our maximum domain so long as the derivatives are greater than $0$ (we will show this in the next section). Then we shift by an $\omega_k$ so that $e \up^k 0 = \Phi(\omega_k) + \Lambda(\omega_k) = 1$. This will tell us that if $\mathcal{A}_{k-1}$ is defined on $\mathbb{R}$ that $e \up^k t$ is defined on $\mathbb{R}$. If $\mathcal{A}_{k-1}$ is defined on $(-\alpha_{k-1},\infty)$ then $e\up^k t$ is defined on $(-\alpha_k,\infty)$. This will be more thoroughly explained in the appendix where we run over the behaviour of these hyper-operators more so. If the reader is confused by this I suggest skipping briefly to Appendix \ref{appA}.

\section{Showing infinite differentiability generally}\label{sec7}
\setcounter{equation}{0}

Now we're going to rinse and repeat the last time we had to show infinite differentiability. Proving $e \up^k t$ is $\mathcal{C}^{\infty}$ is very similar to showing $e \tet t$ was $\mathcal{C}^\infty$. It's very much the same proof; but it requires some subject changes. First of all,

\[
\sum_{m=1}^\infty ||\Lambda_{m+1}(t) - \Lambda_m(t)||_{t \ge T} \le \infty
\]

Second of all, $\Lambda$ is the solution to the equation,

\[
\Lambda(t) = \mathcal{A}_{k-1}(\Phi(t+1) + \Lambda(t+1)) - \Phi(t)\\
\]

We're going to talk about the sequence of terms $\Lambda_m$; which are infinitely differentiable, and approach $\Lambda$ in a uniform manner for $[T,\infty)$. For each derivative $l$, there's a function $F_l(t,x_0,x_1,..,x_l)$ such that,

\[
\Lambda_m^{(l)}(t) = F_l(t+1,\Lambda_{m-1}(t+1),\Lambda_{m-1}'(t+1),...,\Lambda_{m-1}^{(l)}(t+1))\\
\]

And $F_l$ is generated by the same recursion as before,

\begin{eqnarray*}
F_0(t,x_0) &=& \mathcal{A}_{k-1}(\Phi(t) + x_0) - \Phi(t-1)\\
F_{l+1}(t,x_0,...,x_{l+1}) &=& \frac{\partial F_{l}}{\partial t} + \sum_{d=0}^{l} x_{d+1}\frac{\partial F_l}{\partial x_d}\\
\end{eqnarray*}

Where by induction we can show,

\[
\frac{\partial F_l}{\partial x_d} \to 0\,\,\text{as}\,\,t\to\infty\\
\]

\[
\frac{F_l(t,\mathbf{x})}{F_l(t,\mathbf{x}')} \to 1\,\,\text{as}\,\,t\to\infty\\
\]

So that, our functions $\Lambda^{(m)}(t)$ satisfy the formula,

\begin{eqnarray*}
\Lambda_{m+1}^{(l)}(t) - \Lambda_{m}^{(l)}(t) &=& F_l(t+1,\Lambda_{m}(t+1),\Lambda_{m}'(t+1),...,\Lambda_{m}^{(l)}(t+1))\\ &-&F_l(t+1,\Lambda_{m-1}(t+1),\Lambda_{m-1}'(t+1),...,\Lambda_{m-1}^{(l)}(t+1))\\
\end{eqnarray*}

We are going to go by induction, so we'll assume that $\Lambda_m^{(j)}(t)$ converges for $j<l$. Then really, what we want, is that the above expression is summable because $\Lambda_m$ is normal as $t$ grows; and then the mean-value theorem takes care of the rest.

The function $\Lambda(t)$ looks a lot like $\Lambda_m$ for large $m$. If this is summable convergence, than proving that $\Lambda_{m}^{(l)}(t)$ converges in the same manner reduces to a conversation about $F_l(t,x_0,x_1,...,x_l)$. The result to be shown is,

\[
\sum_{m=1}^\infty ||\Lambda_{m+1}^{(l)}(t) - \Lambda_m^{(l)}(t)||_{t \ge T}  < \infty\\
\]

In summation,

\begin{eqnarray*}
\sum_{m=1}^\infty ||\Lambda_{m+1}^{(j)}(t) - \Lambda_m^{(j)}(t)||_{t \ge T}  < \infty \,\,\text{for}\,\,0 \le j< l\,\,&\Rightarrow&\\
\sum_{m=1}^\infty ||\Lambda_{m+1}^{(l)}(t) - \Lambda_m^{(l)}(t)||_{t \ge T}  < \infty\\
\end{eqnarray*}

Now we could almost directly copy the method we used to prove $\tau$ is infinitely differentiable to prove $\Lambda$ is infinitely differentiable. The trouble is something that really throws a stick in our gears. The function $\tau(t)$ looks like $t$, and so did its derivatives, causing either vanishing or approaching $1$. The function $\Lambda$ satisfies a different form,

\begin{eqnarray*}
|\Lambda_m(t) - \Lambda_1(t)| &\le& \sum_{c=1}^m \frac{t+c}{\prod_{j=1}^c \Phi(t+j)}\\
|\Lambda(t) - \Lambda_1(t)| &\le& \sum_{m=1}^\infty |\Lambda_{m+1} - \Lambda_m|\\
&\le& \sum_{m=1}^\infty \frac{t+m}{\prod_{j=1}^m \Phi(t+j)}\\
\end{eqnarray*}

Which means, $\Lambda$ looks like $\Lambda_1$ for large $t$. And the derivatives of $\Lambda$ look like the derivatives of $\Lambda_1$ for large $t$. So we can almost run the same procedure we ran before, but it's going to be inherently more difficult.

A couple of things to remember is that these functions $\mathcal{A}_{k-1}$ are sub-logarithmic (they grow far slower than the logarithm). So we can compare $\Phi$ to $\phi$ from before. \c{C}a c'est,

\[
\mathcal{A}_{k-1}(a\cdot b) \le \mathcal{A}_{k-1}(a) + \mathcal{A}_{k-1}(b)\\
\]

Of which,

\[
\Lambda(t) \le \mathcal{A}_{k-1}(\Phi(t+1)) - \Phi(t) + \mathcal{A}_{k-1}(1+\frac{\Lambda(t+1)}{\Phi(t+1)})\\
\]

Now $\Phi(t+1) = e^t e \up^{k-1} \Phi(t) \le e \up^{k-1} \Phi(t) + t$, which gives,

\[
\Lambda(t) \le t + \mathcal{A}_{k-1}(1+\frac{\Lambda(t+1)}{\Phi(t+1)})
\]

This looks like the original case. The difference is, if we want to exploit everything to the full potential; it's better if we write,

\[
\lim_{t\to\infty} \frac{\Lambda(t)}{\Lambda_1(t)} = 1\\
\]

And by a L'H\^{o}pital/Bernoulli argument--it works for higher order derivatives.

\[
\lim_{t\to\infty} \frac{\Lambda^{(l)}(t)}{\Lambda_1^{(l)}(t)} = 1\\
\]

Now, the take away is similar as before. The iteration we've defined is,

\[
\Lambda_m^{(l)}(t) = \OmSum_{c=1}^m F_l(t+c,\Lambda_{m-c}(t+c),...,\Lambda^{(l-1)}_{m-c}(t+c),x)\bullet x \Big{|}_{x=0}\\
\]

But we want to look at each $\Lambda_m^{(l)}$ like $\Lambda_1^{(l)} + \mathcal{O}(\frac{1}{e^t})$ for large $t$. Where the error still drops off well enough so that we can continue with the proof that we used on $\tau$. The trouble being, we'll have to account for another function somewhere in the mix.

Our goal then; show this iteration converges. The procedure will be similar as to how it was done with $\tau$, but it will require some finesse on our part.

\begin{theorem}
The function $\Lambda$ is infinitely differentiable.
\end{theorem}

\begin{proof}
We will assume that $\Lambda$ is $j$ times differentiable for $j < l$; and that, for $T>0$ large enough,

\[
\sum_{m=1}^\infty ||\Lambda_{m+1}^{(j)}(t) - \Lambda_m^{(j)}(t)||_{t\ge T} < \infty\\
\]

And the second fact we'll use is,

\[
\sum_{c=1}^\infty ||F_l(t+c,\Lambda_{m-c}(t+c), \Lambda'_{m-c}(t+c),...,\Lambda_{m-c}^{(l)}(t+c)) - \Lambda^{(l)}_1(t+c)||_{t \ge T}<\infty \\
\]

Which is equivalent to,

\[
\sum_{c=1}^\infty ||\Lambda_m^{(l)}(t+c) - \Lambda^{(l)}_1(t+c)||_{t \ge T}<\infty \\
\]

Which follows for all $m$. It's routine (at this point in the paper) to show this is a normal sequence. This is essentially to say,

\[
||\OmSum_{c=1}^m F_l(t+c,\Lambda_{m-c}(t+c),...,\Lambda^{(l-1)}_{m-c}(t+c),x)\bullet x - \Lambda_1^{(l)}(t)||_{t \ge T, 0 \le |x-\Lambda_1^{(l)}(t+m)| \le 1} < M\\
\]

We can strengthen this to $||\Lambda_m^{(l)}(t) - \Lambda_1^{(l)}(t)||_{t\ge T} \le 1$, so long as we choose $t \ge T$ large enough. And, now we can almost continue like in the proof of $\tau$. There are numbers $0 < \lambda_0,...,\lambda_l < 1$ such for large enough $T$,

\[
||\frac{\partial F_l}{\partial x_j}(t,x_1,...,x_{j-1},x_j)||_{t \ge T, 0 \le |x_j-\Lambda_1^{(j)}(t)| \le 1} \le \lambda_j\\
\]

Now,

\begin{eqnarray*}
||\Lambda_{m+1}^{(l)}(t) - \Lambda_m^{(l)}(t)||_{t\ge T} &=& ||F_{l}(t+1, \Lambda_m(t+1),...,\Lambda^{(l-1)}_{m}(t+1),\Lambda_m^{(l)}(t+1))-\\
&-& F_{l}(t+1, \Lambda_{m-1}(t+1),...,\Lambda^{(l-1)}_{m-1}(t+1),\Lambda_{m-1}^{(l)}(t+1))||\\
&\le& \sum_{j=0}^l \lambda_j||\Lambda^{(j)}_m(t+1) - \Lambda^{(j)}_{m-1}(t+1)||_{t\ge T}\\
\end{eqnarray*}

Which when iterated derives as,

\begin{eqnarray*}
||\Lambda_{m+1}^{(l)}(t) - \Lambda_m^{(l)}(t)||_{t\ge T} &\le& \sum_{c=0}^{m-2}\sum_{j=0}^{l-1} \lambda_j \lambda_l^{c}||\Lambda^{(j)}_{m-c}(t+c+1) - \Lambda^{(j)}_{m-c-1}(t+c+1)||_{t\ge T}\\
&+& \lambda_l^{m-1}||\Lambda_2^{(l)}(t+m) - \Lambda_1^{(l)}(t+m)||_{t\ge T}\\ 
\end{eqnarray*}

The above expression is summable by the induction hypothesis. Completing the proof.
\end{proof}

The last thing we need to complete our construction of $e \up^k t$ is,

\[
\frac{d}{dt}e \up^k t > 0\\
\]

Now for large $t > T$ this is true. This is because $\Lambda'(t)  \to \frac{d}{dt}\mathcal{A}_{k-1} \Phi(t+1) - \Phi'(t)$. And $\frac{d}{dt}\mathcal{A}_{k-1} \Phi(t+1) > 0$ and $\Phi'(t) > 0$. So for large enough $t > T$, we know $\Phi'(t) + \Lambda'(t) > 0$. The rest is by infinite descent. Insofar as,

\[
0 < \frac{d}{dt} e \up^k t = \frac{d}{dt} e \up^{k-1} e \up^k (t-1) = \big{(}\frac{d}{dx} e \up^{k-1} x \big{)} \frac{d}{dt} e \up^k (t-1)\\
\]

Which, implies that $\frac{d}{dt} e \up^k (t-1) > 0$ because $\big{(}\frac{d}{dx} e \up^{k-1} x \big{)} > 0$ everywhere. Therefore by infinite descent, the result.

\section{The Amalgamation}

We state our very dense theorem here. It is a procedural mechanism to produce $\mathcal{C}^\infty$ hyper-operations. We state it in a recursive format, to knock home the recursive nature.

\begin{theorem}
Define $e \up^1 t = e^t$, there are a sequence of numbers $\omega_k$ such that the following recursion,

\[
\Phi_k(t) = \OmSum_{j=1}^\infty e^{t-j}e \up^k x \bullet x \Big{|}_{x=0}\\
\]

And,

\[
e \up^{k+1} t = \lim_{n\to\infty} \mathcal{A}_{k} \mathcal{A}_k \cdots (n\,\text{times})\cdots \mathcal{A}_k \Phi_k(t+n+\omega_k)\\
\]

Where $\mathcal{A}_k = \big{(}e \up^k t\big{)}^{-1}$, satisfies the following properties:

\begin{enumerate}
\item $e \up^k 0 = 1$\\
\item If $k$ is even, $e \up^k t : (\alpha_k,\infty) \to \mathbb{R}$ bijectively for some $-k\le \alpha_k \le 1-k$\\
\item If $k$ is odd, $e \up^k t : \mathbb{R} \to (\alpha_k, \infty)$ bijectively\\
\item $e \up^k t$ is infinitely differentiable in $t$\\
\item $\frac{d}{dt}e \up^k t > 0$\\
\item $e\up^k t$ is a diffeomorphism\\
\item $e \up^{k-1} e \up^k t = e \up^k (t+1)$\\
\end{enumerate}
\end{theorem}

\section*{In Conclusion}\label{sec12}
\setcounter{equation}{0}

We have sketched the construction of a sequence of hyper-operator functions $e \up^k t$. It is necessary to describe these solutions more acutely. The author knows no method of constructing a uniqueness criterion for this sequence, or even each function individually. Though, he suspects the \textit{eventual} monotonicity of each derivative may suffice--perhaps hinging on some extraneous growth condition. We haven't proved this property here; the author considers what he has an insufficient theorem.

It would be nice to have monotonicity on all of $\mathbb{R}^+$ for all $\frac{d^n}{dt^n}e \up^k t$--not just eventually for large enough $t>T$.  \textit{Can the reader prove this: For all $n$ there exists $T$ such for $t>T$ the function $\frac{d^n}{dt^n}e \up^k t > 0$? If they can, can they prove the stronger result that $\frac{d^n}{dt^n}e \up^k t > 0$ for all $t \in \mathbb{R}^+$? The author would be eternally grateful.}

We do not know much about these hyper-operations. The conjecture is that they're no-where analytic; I had misstepped earlier when I thought tetration would be analytic. I'd like to thank Sheldon Levenstein for pointing out this probably doesn't happen.

The author doesn't know a great deal about these questions; but he hopes to find out; and hopes others do as well.

\appendix
\section{Basic Results About Hyper-operations}\label{appA}

The first thing we're going to put in this appendix is something that is well known amongst people who have fiddled and manipulated hyper-operators on the real-line. But to those of you where all of this is novel, it may not be apparent. We start with the first hyper-operator, the exponential function,

\[
e \up x = e^x\\
\]

Now this is $\mathcal{C}^{\infty}$ on the entire real line, and it approaches $0$ as $x \to -\infty$. Further, $e^x : \mathbb{R} \to \mathbb{R}^+$ bijectively. So it has an inverse function $\log(x) : \mathbb{R}^+ \to \mathbb{R}$. Now, when we construct tetration, something interesting happens.

Since we have normalized these hyper-operators so that $e \up^k 0 =1$ we know that,

\[
e \tet -1 = \log(e \up^2 0) = 0\\
\]

And then, $e \tet -2 = \log(0) = -\infty$. So we know that, $e \tet x : (-2,\infty) \to \mathbb{R}$ bijectively. Now, it's inverse on the other hand, $\slog : \mathbb{R} \to (-2,\infty)$ bijectively. So now, when we define pentation,

\begin{eqnarray*}
e \up\up \up 0 &=& 1\\
e \up\up\up -1 = \slog(1) &=& 0\\
e \up \up \up -2 = \slog(0) &=& -1\\
\end{eqnarray*}

And as we continue, since there are no singularities to worry about with iterated $\slog$, we know that $e \up\up\up x : \mathbb{R} \to (\alpha,\infty)$ bijectively; where $\alpha$ is the limit of $\lim_{n\to\infty} \slog^{\circ n}(0)$. And therefore its inverse $\mathcal{A}_3 : (\alpha,\infty) \to \mathbb{R}$ bijectively. And at this point, we are back to the exponential case. It takes $\mathbb{R} \to$ a line equivalent to $\mathbb{R}^+$.

This will loop indefinitely, where the take away is,

\begin{eqnarray*}
e \up^k 0 &=& 1\\
e \up^k -1 &=& 0\\
e \up^k -2 &=& -1\\
&\vdots&\\
e \up^k -j &=& 1-j\,\,\text{for}\,\,0\le j<k\\
\end{eqnarray*}

And then we will either hit a singularity, or we will keep on going on off to $-\infty$. And we define a correspondence,

\begin{eqnarray*}
e \up^k x\,\, \text{continuous on}\,\,\mathbb{R}\,\,&\text{XOR}&\,\,\mathcal{A}_{k}\,\,\text{continuous on}\,\,\mathbb{R}\\
e \up^k x\,\, \text{continuous on}\,\,(\alpha_k,\infty)\,\,&\text{XOR}&\,\,\mathcal{A}_{k}\,\,\text{continuous on}\,\,(\alpha_k,\infty)\\
\end{eqnarray*}

For some $\alpha_k$. And the initiation of this idea is that for odd $k$, $e \up^k x$ is continuous on $\mathbb{R}$, and its inverse is continuous on $(\alpha_k,\infty)$. And for even $k$, $e \up^k x$ is continuous on $(\alpha_k,\infty)$, and its inverse is continuous on $\mathbb{R}$. Here $1-k \ge \alpha_k \ge -k$.

Now this exact configuration is inherent to hyper-operators as a class of objects. The moment we have a construction,

\begin{enumerate}
\item $e\up^{k} x$ is continuous and monotone\\
\item $e \up^k (x+1) = e \up^{k-1} e \up^k x$\\
\item $e \up^1 x = e^x$ and $e \up^k 0 =1$\\
\end{enumerate}

We open ourselves up to the above analysis. So as it stands, it's inherent to any solution, not just ours. This is why we kept this in the appendix; it is something someone familiar with the object would know; does not appear entirely obvious; it is not particular to our construction.

The irony being, it is very difficult to track down a single source stating this. The author learned this by way of Henryk Trappman; and a post on his tetration forum \cite{Trp}. He described much of this discussion; particularly the proof of successorship $e \up^k -j = 1-j$ for $0 \le j  < k$.

\section{Acknowledgements}\label{appB}

Much of this paper was written with implicit knowledge of what these things will look like. Much of it was written with foreknowledge of what will work, and what will not. To call this intuition is a bit of a misnomer. Much of this was learned by way of Henryk Trappman's Tetration forum \cite{Trp}.

There are years worth of discussions which have happened on this forum. Attempts made, attempts lost, attempts upon attempts by yours truly--and a whole swath of characters. Where each undoubtedly failed attempt, or didn't work quite as well as expected attempt, was a great learning experience for the author--and he knows he's not alone.

For this paper, the author would particularly like to thank Sheldon Levenstein. Not only did he point out that it was highly doubtful that this construction of tetration was holomorphic; he seemed still pleased with a $\mathcal{C}^\infty$ solution. Enough so, that the author felt it worth the while to take this paper in a different direction. Despite that, probably none of these things are analytic; there still seemed to be a point to this construction. (The author typically only thinks a thing worth his time if it's analytic/holomorphic.)

He'd like to whole-heartedly thank everyone at the forum for years of patience and curiousity driven discussions.\\

Regards, James

\end{document}